\documentclass[reqno]{amsart}
\usepackage{amssymb}
\usepackage{hyperref}

\advance\voffset-1truecm\relax
   \advance\hoffset-1truecm\relax
\newtheorem {theorem}{Theorem}[section]

\newtheorem {lemma}{Lemma}[section]

\newtheorem {remark}{Remark}[section]
\newtheorem {proposition}{Proposition}[section]

\usepackage{color}

   \def\R{\mathbb R}

\def\n{\noindent}
\def\nhd{neighborhood}

\def\Om{\Omega}
\def\i{\sl{\bold i}}

\def\va{\varphi}

\def\ga{\gamma}

\def\la{\lambda}

\def\ve{\varepsilon}

\def\al{\alpha}
\def\bt{\beta}
\def\la{\lambda}

\def\fr{\frac}
\def\om{\omega}
\def\De{\Delta}

\def\de{\delta}

\def\al{\alpha}

\def\R{\mathbb R}
\def\vol{\text{vol}}

\def\ees{{\accent"5E e}\kern-.385em\raise.2ex\hbox{\char'23}\kern-.08em}
\def\EES{{\accent"5E E}\kern-.5em\raise.8ex\hbox{\char'23 }}
\def\ow{o\kern-.42em\raise.82ex\hbox{
\vrule width .12em height .0ex depth .075ex \kern-0.16em \char'56}\kern-.07em}
\def\OW{O\kern-.460em\raise1.36ex\hbox{
\vrule width .13em height .0ex depth .075ex \kern-0.16em \char'56}\kern-.07em}

\large

\title{Volume estimates of sublevel sets of real polynomials}

\author{Nguyen Quang Dieu, Dau Hoang Hung, TI\EES N S\OW N PH\d{A}M, Hoang Thieu Anh}
\address{Hanoi National University of Education, Hanoi, Vietnam}
\email{dieu$\_$vn@yahoo.com}
\thanks{Corresponding author: Nguyen Quang Dieu (Ha Noi National University of Education)}

\address{Phan Boi Chau High school for Gifted students, Vinh, Nghe An, Vietnam}
\email{dhhungk9@yahoo.com}

\address{Department of Mathematics, University of Dalat, 1 Phu Dong Thien Vuong, Dalat, Vietnam}
\email{sonpt@dlu.edu.vn}

\address{Faculty of Fundamental Sciences, University of Transport and Communication, Hanoi, Vietnam}
\email{hoangthieuanh@gmail.com}
\large


\date{\today}

\begin{document}
\maketitle
\begin{abstract}
We give upper bounds for volume of sublevel sets of  real polynomials. Our method is to combine a version of global \L ojasiewicz inequality with some well known estimates on volume of tubes around real algebraic sets.
Some applications to oscillatory integrals and integration indices of real polynomial are also given.
\end{abstract}
\section{Introduction}

Let $P$ be a real polynomial on $\mathbb R^n$. For $\de \in \R$ we define the sublevel set
$V_\de:=\{x \in \R^n: \vert P(x)\vert \le \de\}.$
The main goal of this article is to study estimates from above for $\vol_n (V_\de \cap \Om),$ where $\Om$ is a bounded open  subset of $\R^n$ and $\vol_n$ is the $n-$dimensional Lebesgue measure of $\R^n$.
For technical reason, we only restrict ourself to the cases where $\Om=\De_r^n:=(-r, r)^n.$
In the case where $n=1,$ we have the following result which is a special case of \cite[Proposition 2.1]{Carbery1999}:
Let $P(x)=a_0x^d+\cdots$ be a polynomial in $\mathbb R$ of degree $d \ge 1.$ Then there exists a constant $C_d>0$ depending only on $d$ such that 
\begin{equation} \label{onedim}
\vol_1 \{\vert P(x)\vert \le \de\} \le C_d \frac{\de^{1/d}}{\vert a_0\vert^{1/d}}, \ \forall \de>0.
\end{equation}
Using (\ref{onedim}) and an inductive argument, we obtain the following generalization of the above estimate 
(see \cite[Theorem 7.1]{Carbery1999}): Let $P$ be a polynomial in $\mathbb R^n$ of degree $d \ge 1$ and satisfies
$\frac{\partial^{\vert \beta\vert} P}{\partial x^\bt} \ge 1$ on the cube $[0, 1]^n$ for some multi-index $\bt$. Then
there exists a constant $C_{d,n}>0$ depending only on $d, n$ such that 
\begin{equation} \label{highdim}
\vol_n \{x \in [0,1]^n: \vert P(x)\vert \le \de\} \le C_{d,n} \de^{1/\vert \bt\vert}, \ \forall \de>0.
\end{equation}
 The above result can be used in our problem as follows.
Write $P(x)=\sum_{\vert \al \vert \le d} a_\al x^\al$, then for some multi-index $\al=(\al_1, \cdots, \al_n)$ with 
$\vert \al\vert:=\vert \al_1\vert+\cdots+\vert a_n \vert=d$
and $a_\al \ne 0,$ we set
$$Q(y):= \frac1{(2r)^d\al!a_\al} P(2ry-r), y \in \R^n,$$
where $\al!=\al_1 !\cdots \al_n!.$ 
Then $Q$ is a polynomial of degree $d$ and $\frac{\partial^{\vert \al\vert} Q}{\partial y^\al}=1$ on $\R^n.$ 
So by (\ref{highdim}), we get
$$\vol_n \{y \in [0, 1]^n: \vert Q (y)\vert \le \de\} \le C_{d,n}\de^{1/d}, \ \forall \de>0.$$
By considering the change of variable $x:= 2ry-r$ we infer
\begin{equation}\label{CCW}
\vol_n (V_\de \cap \De_r^n) \le \fr{C'_{d,n}}{(\al! \vert a_\al\vert)^{1/d}}r^{n-1} \de^{1/d}, \ \forall \de>0
\end{equation}
There is also another route to bound $\vol_n (V_\de \cap \De_r^n)$ from above by using estimates of the sup-norm of $P$
over compact sets of $\mathbb R^n.$
This idea has been exploited in \cite{B-G1973} where the authors provide a multi-dimensional version of Remez's inequality for polynomial functions on $\R^n.$
More precisely, for  an open convex domain $\Om$ in $\mathbb R^n$ and a measurable subset $\om \Subset \Om$ we obtain
the following "doubling" inequality
(see Theorem 2 in \cite{B-G1973}, especially (8) in p. 350)
\begin{equation}
\Vert P \Vert_\Om \le \Big ( \frac{4n \vol_n (\Om)}{\vol_n (\om)}\Big )^d \Vert P \Vert_\om,
\end{equation}
where $d$ is the degree of $P$.
It follows (see \cite{B-G1973} p. 354) that  the volume of $\om:=V_\de \cap \Om$ may be estimated from above in terms of $\de, d, n$ and $\Vert P\Vert_\Om.$ 
Since the last quantity is somewhat imprecise, it is natural to have some bounds which depends only on $n$ and the polynomial $P$.

In this work, we approach the above problem in a different way. More precisely, we compare the sublevel set $V_\de$ with tubular \nhd s around the real algebraic set $Z(P):=\{x \in \R^n: P(x)=0\}.$
This is done with the help of  a global \L ojasiewicz inequality for polynomials (Lemma 3.1).
Since the volume of such a tube is known to be estimated in terms of the radius and the degree of the defining polynomials (cf. Theorem 2.1), we obtain some
upper bound for $\vol_n (V_\de \cap \De_r^n)$ (cf. Theorem 3.1).
We should say that similar (but weaker)  global \L ojasiewicz inequalities were obtained earlier (see \cite[Theorem 4.6]{Dinh2012-1}) and \cite[Theorem 2.1]{HNP2013}).
One novelty in our work is to provide a sharper \L ojasiewicz inequality where the notion of {\it admissible monomial} is taken into account (cf. Lemma 3.1).
The other main result is Theorem 3.2 where we employ the change of variable formula repeatedly to get  improvement in estimating $\vol_n (V_\de \cap \De_r^n)$ in certain circumstances.

Finally, we apply our volume estimates to give explicit lower bounds for integration indices of polynomials and finiteness of oscillatory integrals.
Recall that for a real polynomial $P$ in $\R^n$ with $P(0)=0$, the integration index $\i (P)$ is defined to be the supremum of all $t>0$ such that $\vert P\vert^{-t}$ is integrable on {\it some} \nhd\ of $0$.
By a standard rearrangement formula, we can see that the finiteness of $\int_{\De_r^n} \vert P\vert^{-t}d\la_n$ depends on how fast the volume of $\De_r^n \cap \{\vert P \vert \le \de\}$ decreases to $0$ as $\de \to 0$ (cf. Proposition 4.1 and Proposition 4.2).
Likewise, by a classical inequality of Van der Corput (cf. Lemma 4.3) we may bound the  integral of $e^{-i \lambda P}$ over the part where the partial derivatives of $P$ are bounded from below.
Then the oscillatory integral $\int_{\De_r^n} e^{-i \lambda P}d\la_n$ may be estimated from above (cf. Proposition 4.3)
by using the volume estimates for sublevel sets of partial derivatives of $P$.
\vskip0,6cm
\noindent
{\bf Acknowledgments.} We are grateful to the referee for his (her) useful comments that helped to improve our exposition significantly.
\section{Preliminaries}
Throughout this paper, ${\Bbb R}^n$ denotes Euclidean space of dimension $n.$
 The corresponding inner product (resp., norm) in ${\Bbb R}^n$  is defined by $\langle x, y \rangle$ for any $x, y \in {\Bbb R}^n$ (resp., $\| x \| :=
\sqrt{\langle x, x \rangle}$ for any $x \in {\Bbb R}^n$). For a subset $X \subset \R^n$ and $a \in \R^n$, we let
$d(a, X):= \inf_{x \in X} \Vert x-a\Vert.$
For $\de>0$, the tubular \nhd\ of $X$ with radius $\de$ is defined by
$X_\de:= \{x \in \Bbb R^n: d(x, X)<\de\}.$

We recall the following basic result of Wongkew \cite{Wongkew1993} on upper bound of tubular neighborhoods of real algebraic varieties.
\begin{theorem}
Let $m$ be the codimension of a real algebraic variety $Z \subset \R^n$ whose defining polynomials are all bounded in degree by $d$.
Let $S$ be an arbitrary ball of radius $r$ in $\R^n$. Then there exist positive constants $c_m, \cdots, c_n$ depending only on $n$ such that for all positive $\de$ we have
$$\vol_n (Z \cap S)_\de \le \sum_{j=m}^n c_jd^j\de^jr^{n-j}.$$
\end{theorem}
We should say that a variation of this result for semi-algebraic sets $Z$ was obtained earlier, see Theorem 5.9 in \cite{Yomdin2004}. Nevertheless, in this more general setting, the constants $c_m, \cdots, c_n$ depend also on
the so-called {\em diagram} of $Z$ (see Definition 4.9 in \cite{Yomdin2004}). The next result,  a rather straightforward consequence of Theorem 2.1, will play a key role in our work.

\begin{lemma}
Let $m$ be the codimension of a real algebraic variety $Z \subset \R^n$ whose defining polynomials are all bounded in degree by $d$.
Let $S$ be an arbitrary ball of radius $r$ in $\R^n$. Then there exists a positive constant $C$ depending only on $n$ such that for all positive $\de$ we have
$$\vol_n (Z_\de \cap \De_r^n) \le C((d\de)^m r^{n-m}+(d\de)^n).$$
\end{lemma}
\begin{proof}
By the triangle inequality we have the following inclusions
$$Z_\de \cap \De_r^n \subset (Z \cap \De_{r+\de}^n)_\de \subset (Z \cap \mathbb B (0, \sqrt{n} (r+\de))_\de.$$
By Theorem 2.1 we have the following chain of estimates with $C_1, C_2$ and $c_j$ are constants depend only on $n$
\begin{eqnarray*}
\vol_n (Z \cap \mathbb B (0, \sqrt{n} (r+\de))_\de & \le & \sum_{j=m}^n c_j (d\de)^j  (\sqrt{n})^{n-j} (r+\de)^{n-j}\\
&\le &C_1 [(d\de)^m (r+\de)^{n-m}+ (d\de)^n] \\
& \le &C_1 [2^{n-m-1} (d\de)^m (r^{n-m}+\de^{n-m}) +(d\de)^n]\\
&\le & C_2 [(d\de)^m r^{n-m}+(d\de)^n].
\end{eqnarray*}
Here we use the elementary inequalities
$$a^{n-j}b^j+ a^{m+j} b^{n-m-j} \le a^n+a^m b^{n-m}$$
for $0 \le j \le n-m$ and $a, b>0$ in the second and the third lines.
\end{proof}
We also need the following elementary lemma.
\begin{lemma}
Let $a_1, \cdots, a_n$ be positive numbers and $n, l \ (n \ge 2)$ be positive integers. Let $f$ be the function defined on  ${(\R^+)}^n$ by
$$f(x_1,\cdots, x_n)=\frac{a_1}{x_1}+a_2\frac{x_1}{x_2}+\cdots+a_n \frac{x_1 \cdots x_{n-1}}{x_n}+ x_n^{1/l}.$$
Then
$$\min_{(\R^+)^n} f=C(n, l) \va (a_1, \cdots, a_n)^{\frac1{2^{n-1}+l}},$$
where $C(n, l)>0$ depends only on $n, l$ and  $$\va (a_1, \cdots, a_n):= a_1^{2^{n-2}} a_2^{2^{n-3}} \cdots a_{n-1} a_n.$$
Moreover, the minimum of $f$ is realized at some point $(x^0_1, \cdots, x^0_n) \in (\R^+)^n$ with
$$x_n^0= \Big (\frac{C(n, l)l}{2^{n-1}+l} \Big )^l \va (a_1, \cdots, a_n)^{\frac{l}{2^{n-1}+l}}.$$
\end{lemma}
\begin{proof}
We write
\begin{align*}
f(x_1,\cdots, x_n) &= \Big [ \underbrace{\frac1{l}x_n^{1/l}+\cdots+\frac1{l}x_n^{1/l}}_{l \ \text{times}}\Big] + a_n \frac{x_1\cdots x_{n-1}}{x_n} +\\
&+\sum_{j=1}^{n-1}\Big [ \underbrace{\frac1{2^{n-j-1}} a_j \frac{x_1 \cdots x_{j-1}}{x_j}+\cdots + \frac1{2^{n-j-1}} a_j \frac{x_1 \cdots x_{j-1}}{x_j}}_{2^{n-j-1} \ \text{times}} \Big].
\end{align*}
By applying the inequality of arithmetic and geometric  means we obtain
$$f(x_1, \cdots, x_n) \ge C(n, l) (a_1^{2^{n-2}} a_2^{2^{n-3}} \cdots a_{n-1} a_n)^\frac1{2^{n-1}+l}.$$
Moreover, the equality occurs at some point $(x_1^0, \cdots, x_n^0) \in (\R^+)^n$ such that for all $j = 1, \ldots, n,$
\begin{eqnarray*}
\frac{C(n, l)}{2^{n-1}+l} \va(a_1, \cdots, a_n)^{\frac1{2^{n-1}+l}} &=&  \frac1{l}{x_n^0}^{\frac{1}{l}} \\
&=& a_n \frac{x_1^0 \cdots x_{n-1}^0}{x_n^0} \\
&=&  \frac{a_j}{2^{n-j-1}} \frac{x_1^0 \cdots x_{j-1}^0}{x_j^0}.
\end{eqnarray*}
These equations can be solved easily and we see that $x_n^0$ is indeed given by the desired formula.
The proof is complete.
\end{proof}
\section{Main results}
We begin with a global \L ojasiewicz inequality which is of interest in its own right.
This result is analogous to Theorem 2.1 in \cite{HNP2013} where only smoothness of functions is taken into account.
The following notions related to a polynomial $P \in \mathbb R[x_1, \cdots, x_n]$ are crucial in our work. 
Let $d:=\deg P$, then we write
$$P(x)=\sum_{\{\al \in \mathbb N^n: \vert \al\vert \le d\}} a_\al x^\al.$$
We say that a monomial $x^\al=x_{1}^{\al_1} \cdots x_{n}^{\al_n}, \al=(\al_1, \cdots, \al_n)$ is {\it admissible} with respect to $P$ if $a_\al \ne 0$ and there exists a permutation $(i_1, \cdots, i_n)$ of $(1, \cdots, n)$ such that $\al_{i_1} \ge 1$ and that
for any monomial $x_{i_1}^{\beta_1} \cdots x_{i_n}^{\beta_n}$ with 
$(\al_1, \cdots, \al_n) \ne (\beta_1, \cdots, \beta_n)$
and $a_{\beta_1, \cdots, \beta_n} \ne 0$ we can find an index $1 \le j \le n$ such that $\al_{i_j}>\beta_{i_j}$ while
$\al_{i_k}=\beta_{i_k}$ for every $1 \le k \le j-1$
(if there exists an index $j=1$ with $\al_{i_1} >\beta_{i_1}$ then we should not check the other conditions $\al_{i_k}=\beta_{i_k}$ for $1 \le k \le j-1$).
In this case, we also say that $\al$ is an {\it admissible index} for $P$.
For example, if 
\begin{equation} \label{example}
P(x_1,x_2)=x_1^d+x_1x_2^m+x_2^p,
\end{equation}
where $d>m>p\ge 1$ are integers, then $x_1^d$ and $x_1x_2^m$ are admissible monomials of $P$.

Observe that for each polynomial $P$ on $\mathbb R^n$, there exists at least one admissible monomial. Indeed, for $n=1$, the conclusion is clear. Assume the statement is true for $n-1$, we now write $P$ as a polynomial in the variable $x_1,$
$$P(x_1, x')=x_1^{\al_1} P_0 (x')+....$$
Then by the induction hypothesis, there exists an admissible monomial
$x_2^{\al_2} \cdots x_n^{\al_n}$ for $P_0$. It follows that 
$x_1^{\al_1} x_2^{\al_2} \cdots x_n^{\al_n}$ is an admissible monomial for $P$.

The admissible degree of $P$ is defined as
$$ad(P):=\min\{ \vert \al\vert: x^\al \ \text{is an admissible monomial for}\  P\}.$$
Now our global \L ojasiewicz inequality reads as follows.
\begin{lemma}
Let $\al=(\al_1, \cdots, \al_n) \in \mathbb N^n$ and
$x^\al:=x_{1}^{\al_1} \cdots x_{n}^{\al_n}$  be an admissible monomial with respect to $P$.
Then there exists a real algebraic set  $Z_\al$ such that $\dim (Z(P) \cup Z_\al)=n-1$ and
$$d(x, Z(P) \cup Z_\al) \le  \frac{4}{\vert a_\al\vert^{\fr1{\vert \al\vert}}}
|P(x)|^{\frac{1}{\vert \al \vert}}, \quad \forall x \in \mathbb{R}^n.$$
\end{lemma}
\begin{proof}
We use an idea in the proof of the main theorem in \cite{Dinh2013-2} (see also Theorem 2.1 in \cite{HNP2013}).
After a permutation of coordinates, we achieve that $i_j=j$ for every $1 \le j \le n.$
So we can write
\begin{eqnarray*}
P(x) &=& x_1^{\al_1} P_{0, 1} (x')+\cdots+x_1 P_{\al_1-1, 1} (x')+P_{\al_1, 1} (x'), \\
P_{0, 1} (x') &=& x_2^{\al_2} P_{0, 2} (x'')+\cdots+P_{\al_2, 2} (x''), \\
& \vdots& \\
P_{0, n-1} (x_n) &=&a_\al x_n^{\al_n}+\text{lower order terms}.
\end{eqnarray*}
Here $x'=(x_2, \cdots, x_n), x''=(x_3, \cdots, x_n), \ldots,$ and $P_{i, j}$ are polynomials.
Now we set
\begin{eqnarray*}
Z_1 &:=& \{x \in \R^n: \ \frac{\partial P}{\partial x_1} (x) =0\}, \\
Z_2 &:=& \{ x' \in \R^{n-1}: P_{0, 1} (x') \frac{\partial P_{0, 1}}{\partial x_2} (x')=0\}, \cdots, Z_n :=\{x_n \in \R: P_{0, n-1} (x_n) P'_{0, n-1} (x_n)=0\}\\
Z_\al &:=& \bigcup_{k \in K} (\R^{k-1} \times Z_k).
\end{eqnarray*}
Here $K$ is the set of indices $k\ge 1$ such that $Z_k$ is a proper algebraic subset of $\R^{n-k+1}$. We have to show that  $Z_\al$ is the desired real algebraic set.
First, we note that if $\al_1$ is odd then the equation $P(x_1, x')=0$ always has a real root for each $x' \in \R^{n-1}$. Thus $\dim Z(P)=n-1$.
On the other hand, if $\al_1$ is even then $\ \frac{\partial P}{\partial x_1}$ is a polynomial of odd degree with respect to $x_1$. Therefore $\dim Z_1=n-1$. So, in any case we have $\dim (Z(P) \cup Z_\al)=n-1.$
Now we turn to our basic inequality.
Fix $x \in \R^n \setminus (Z(P) \cup Z_\al)$ and let $t:= d(x,Z(P) \cup Z_\al)>0.$
Consider the polynomial
$$Q (\lambda) := P(\lambda, x') \in \mathbb{R}[\lambda].$$
Since $\min \{d(x, Z_1), d(x, Z(P))\} \ge t>0$, the polynomial $Q$ and its derivative $Q'$ have no root on the interval 
$I := [x_{1}-t, x_1 + t].$ Hence $Q$ is monotone and has constant sign on $I.$ Without loss of generality, we may assume that $Q> 0$ and $Q' > 0$ on $I.$ So $Q$ is positive and increasing on $I.$  Therefore
$$\max_{\lambda \in [x_1- t, x_1]} Q(\lambda) = Q (x_1) = P(x).$$
A classical result of Chebyshev [6, p. 31] says that if $q$ is a {\it monic} real polynomial of 
of one variable having degree $d$ then
$$\max_{0 \le t \le 1} \vert q(t)\vert \ge \fr1{2^{2d-1}}.$$ 
It follows that
$$\vert P(x)\vert=\max_{\lambda \in [x_1-t, x_1]} Q(\lambda) = \max_{\lambda \in [x_1- t, x_1]} |Q(\lambda)| \ge \frac{t^{\al_1}}{2^{2\al_1 - 1}} \vert P_{0, 1} (x')\vert.$$
Notice that $d(x'^k, Z_2) \ge t, \cdots$ by repeating the above reasoning we finally obtain
$$\vert P_{0, n-1} (x_n)\vert \ge \frac{t^{\al_n}}{2^{\al_n-1}} \vert a_\al\vert.$$
Putting all this together we get
$$\vert P(x)\vert \ge \fr{\vert a_\al \vert}{2^{2\vert \al\vert}}t^{\vert \al\vert}.$$
The desired estimate follows easily after rearranging the above inequality 
\end{proof}
\begin{remark}{\rm
If $P$ is the polynomial given in (\ref{example}) then by applying Lemma 3.1 to admissible monomials $x_1^d$ and $x_1x_2^m$, we can find a real algebraic set $Z'$ in $\mathbb R^2$ such that  $\dim (Z(P) \cup Z')=1$ and
$$d(x, Z(P) \cup Z') \le  4 \min \{|P(x)|^{\fr1{d+1}}, |P(x)|^{\fr1{m+1}}\}, \quad \forall x \in \mathbb{R}^2.$$
}\end{remark}
We now come to the first main result of the paper.
\begin{theorem}
Let $P$ be a polynomial of degree $d$ on $\R^n$.
Assume that $\al$ is an admissible index for $P$. Then for every $\de>0, r>0$ we have
\begin{align*}
 \vol_n (V_\de \cap  \De_r^n)   &\le C \Big [d\mu(\al)\de^{1/\vert \al\vert}r^{n-1} + (d\mu(\al)\de^{1/\vert \al\vert})^n\Big ] \\
 & \le C' \Big [ \de^{1/\vert \al\vert}r^{n-1} + \de^{n/\vert \al\vert} \Big ].
 \end{align*}
Here $\mu (\al)= \frac{4}{\vert a_\al\vert^{1/\vert \al\vert}}$ and $C>0$ (resp. $C'>0$) is a constant that depends only on $n$ (resp. on $n, d$ and $\vert a_\al\vert^{1/\vert \al\vert}$).
\end{theorem}
 \begin{proof}
For every $\de>0, r>0$, by Lemma 3.1  we have
$$V_\de \cap \De_r^n \subset \{x \in \De_r^n: d(x, Z(P) \cup Z_\al) <\mu(\al) \de^{1/\vert \al\vert}\}.$$
By applying Lemma 2.1 to $Z(P) \cup Z_\al$, while observing that $\text{codim}\ (Z(P) \cup Z_\al)=1$, we obtain
$$\vol_n  \{x \in \De_r^n: d(x, Z(P) \cup Z_\al) <\mu(\al) \de^{1/\vert \al\vert}\} \le C \Big [d\mu(\al)\de^{1/\vert \al\vert}r^{n-1} + (d\mu(\al)\de^{1/\vert \al\vert})^n\Big ].$$
Here $C>0$ depends only on $n$. The first conclusion follows immediately by combining these above assertions.
This assertion also implies trivially the latter one.
The proof is complete.
\end{proof}
\n
{\bf Remark.} Let $P$ be the polynomial given in (\ref{example}).
Then we may apply Theorem 3.2 for the admissible monomial $x_1x_2^n$
to get the following estimate
\begin{equation} \label{estimate2}
\vol_2 (V_\de \cap  \De_r^2) \le C \Big [r\de^{\fr1{n+1}}+\de^{\fr2{n+1}}\Big],
\end{equation}
where $C>0$ does not depend on $r,\de.$ This estimate is obviously sharper than (\ref{CCW}) when $\de<<1$ and $r>>1.$
\vskip0,3cm
\n
The above result is clearly not satisfactory when $\dim Z(P)<n-1$. For example, if $P(x, y)=x^2+y^2$ then $\vol_2 (V_\de \cap \De_r^2) =\pi \de$ for $r>\de^2.$
In this ``degenerate" case, by imposing additional conditions  on the defining polynomial $P$ and geometry of the real algebraic set $Z(P)$
we have the following estimate for $\vol_n (V_\de \cap \De_r^n)$ with sharper exponent of $r$ (but with implicit exponent of $\de$).
First, we introduce the following piece of terminology: A collection of closed semi-algebraic sets $(S_1, \cdots, S_k)$  in $\R^n$ is said to be {\it separated at infinity} if the following assertion holds:

\n
(i) $S :=S_1 \cap \cdots \cap S_k \ne \emptyset;$

\n
(ii) For sequences $\{a_{1, l}\}_{l \ge1} \subset S_1, \ldots, \{a_{k, l}\}_{l \ge 1} \subset  S_k$ satisfying $\vert a_{j,l}\vert \to \infty$ as $l \to \infty$
$\| a_{j, l}-a_{m, l} \| \to 0$ as $l \to \infty,$ for all $1 \le j < m \le k,$ we have
$$\lim_{l \to \infty} d (a_{j, l}, S)= 0.$$
\n
We should say that in the case $k=2$, this property of "separated at infinity" is essentially equivalent to that of "non-asymptotic at infinity" introduced in \cite{HNP2013}.
For simple examples of semi-algebraic sets which are separated at infinity we may take a finite number of real lines in $\R^2$ that pass through the origin.

Now we come to the following variant of Theorem 3.1.
\begin{proposition}
Let $P$ be a polynomial of degree $d$ with $\dim Z(P)=k<n-1$, where $Z(P):=P^{-1} (0).$ Assume that the following conditions hold:

\n
(a) $P$ is a monic polynomial with respect to each variable $x_i, 1 \le i \le n.$

\n
(b) $(Z_1, \cdots, Z_n)$ are separated at infinity,
where $Z_i:=\{x: \frac{\partial P}{\partial x_i} (x)=0\}.$

Then there exist positive constants $\nu, C$ which are independent of $r, \de$ such that for all $\de>0$ small enough we have
$$\vol_n (V_\de \cap \De_r^n) \le Cr^{k'}\de^\nu,$$
where $Z':=\cap_{i=1}^n  Z_i$ and $k':=\dim Z'.$
\end{proposition}
\n
{\bf Remarks.}
(i) It is easy to see that after a linear change of coordinates $T$, a general polynomial $P$ can always be put in the form of (a). Moreover, if the matrix of $T$ is {\it unitary} then
the property (b) is also preserved.

\n
(ii) For an application, consider the two real variables polynomial  $P(x, y):= (y-x^3)^2+(x-y^3)^2$.
Then $P$  is a monic polynomial of degree $6$ with respect to each variable $x$ and $y.$ By direct computations we get
$$Z_1=\{(x, y) \in \R^2: y^3+3x^2y=3x^5+x\}, Z_2=\{(x, y) \in \R^2: x^3+3y^2x= 3y^5+y\}.$$
We claim that $(Z_1, Z_2)$ is separated at infinity.
For this, consider two sequences $(x_n, y_n) \in Z_1, (x'_n, y'_n) \in Z_2$ with $\vert x_n-x'_n \vert \to 0$ and $\vert y_n-y'_n\vert \to 0.$
Assume that $\vert x_n \vert \to \infty,$ then since $(x_n, y_n) \in Z_1$ we have
$$3x_n^2y_n+y_n^3=3x_n^5+x_n.$$
It follows that $\vert y_n\vert \ge \vert x_n\vert$ for $n \ge n_0$. This implies, for $n \ge n_1,$ that
$$\vert y_n\vert >2^{-1/3} \vert x_n\vert^{5/3}+1.$$
Therefore
$$\vert y_n'\vert > 2^{-1/3} \vert x_n\vert^{5/3}.$$
On the other hand, since $(x'_n, y'_n) \in Z_2$ we obtain
$$3x'_n {y'_n}^2+{x'_n}^3=3{y'_n}^5+y'_n.$$
Since $\vert y'_n\vert \to \infty,$ by the same reasoning we get
$$\vert x'_n\vert> 2^{-1/3} \vert y'_n\vert^{5/3}, \forall n \ge n_2.$$
Putting the two last inequalities together we obtain
$$\vert x'_n\vert> 2^{-8/9}\vert x_n\vert^{25/9},  \forall n \ge n_3.$$
This contradicts the fact that $\vert x_n- x'_n\vert \to 0$. Thus the sequence $\{x_n\}$ is bounded. It follows that $\{y_n\}, \{x'_n\}, \{y'_n\}$ are all bounded as well.
Therefore $(Z_1, Z_2)$ is separated at infinity. Since $\dim Z=\dim Z'=0$, we may apply Proposition 3.1 to get the following estimate
$$\vol_2 (V_\de \cap \De_r^2) \le C \de^\al$$
for $r>0$ and $\de>0$ small enough, where $\al>0$ is a constant that does not depend on $r, \de.$

The key ingredient in the proof is the following fact which may be of independent interest.
\begin{proposition}
Let $(S_1, \cdots, S_k)$ be a collection of closed semi-algebraic sets in $\R^n$ which are separated at infinity.
Then there exist positive constants  $\theta, C, \de$ such that
$$Cd(x, S)^\theta \le d(x, S_1) + \cdots+d(x, S_k) \  \text{if}\  \max \{d(x, S_1), \cdots, d(x, S_k)\}<\de.$$
\end{proposition}
\begin{proof}
We use an argument similar to that of Lemma 3.6 in \cite{Dinh2012-1}.
Since $S_1, \ldots, S_k$ are semi-algebraic we infer that
$$f(x):=d(x, S) \quad \textrm{ and } \quad g(x):=d (x, S_1)+\cdots+d(x, S_k)$$
are semi-algebraic functions on $\R^n.$ We claim that
 for every sequence $\{a_l\}_{l \ge 1} \subset \R^n$ satisfying $g(a_l) \to 0$ we have $f(a_l) \to 0.$
 Then we choose $\{a_{1, l}\}_{l \ge1} \subset S_1, \cdots, \{a_{k, l}\}_{l \ge 1} \subset  S_k$
  such that
 $$\|a_{j, l}-a_l \| = d(a_l, S_j), \quad \forall 1 \le j \le k.$$
 It follows that
 $\|a_{j, l} -a_{m, l} \| \le d(a_l, S_j)+d (a_l, S_m) \to 0$  as $l \to \infty$
 for every $1 \le j<m \le k.$
 This implies, in view of the assumption (b) that $\lim_{l \to \infty} d(a_{j, l}, S)= 0$ for every $1 \le j \le k.$ So we get $f(a_l) \to 0$
 and the claim follows.
 Now we set
 $$\mu (t):= \sup \{f(x) : g(x) = t\} \quad \textrm{ for } \quad t \ge 0.$$
 It is well known that the function $\{t \ge 0\} \rightarrow \mathbb{R}, t \mapsto \mu(t),$ is semi-algebraic. Moreover, by the above reasoning we have
$\lim_{t \to 0} \mu (t) = 0.$
Thus we may apply the growth dichotomy theorem (see \cite{Bochnak1998}) to reach the desired conclusion.
 \end{proof}
\begin{proof} (of Proposition 3.1)
From the assumption $(a)$ and Lemma 3.1 (see the remark that follows the lemma) we can find a constant $C_1>0$ such that
$$d(x, Z(P) \cup Z_i) \le C_1\vert P(x)\vert^{\frac1{d}} \ \forall 1 \le i \le n.$$
Since $\dim Z(P)<n-1$, by the implicit function theorem we have $Z(P)\subset Z_i$ for every $i$. Therefore
$$V_\de \subset (Z_i)_{C_1 \de^{1/d}}, \  \forall \de>0,  \forall 1 \le i \le n.$$
It follows that
\begin{equation}
V_\de \cap \De_r^n \subset \Big (\bigcap_{i=1}^n (Z_i)_{C_1 \de^{1/d}}\Big ) \cap \De_r^n, \  \forall \de>0.
\end{equation}
Since the collection of semi-algebraic sets $(Z_1, \cdots, Z_n)$ is separated at infinity,
using Proposition 3.2 we get positive constants $C_2, \theta, \de_0$ such that
$$C_2 d(x, Z')^\theta \le d(x, Z_1)+\cdots+d(x, Z_n) \  \text{if}\  \max \{d(x, Z_1), \cdots, d(x, Z_n)\}<\de_0.$$
Therefore
\begin{equation}
\Big (\bigcap_{i=1}^n (Z_i)_{C_1 \de^{1/d}}\Big ) \cap \De_r^n \subset Z'_{C_3 \de^{\frac1{\theta d}}} \cap \De_r^n, \ \forall \de<\de_1:=(\de_0/C_1)^d,\end{equation}
where $C_3=(nC_1/C_2)^{1/\theta}.$
Applying Theorem 5.9 in \cite{Yomdin2004} we obtain
\begin{equation}
\vol_n (Z'_{C_3 \de^{\frac1{\theta d}}} \cap \De_r^n) \le C_4 r^{k'} \de^{\frac{n-k'}{\theta d}}. \end{equation}
Putting (7), (8) and (9) all  together we compete the proof.
\end{proof}
We now come to the second main result of the paper.
\begin{theorem} Let $n \ge 2$ and $P$ be a polynomial in $\R^n.$ Define inductively the following real polynomials
 $$P_1=P, P_m (x):= \det \Big (\frac{\partial P_i}{\partial x_j} \Big )_{1 \le i, j \le m-1}, 2 \le m \le n+1.$$
 Assume that $\deg P_{l+1} \ge 1$  for some $1 \le l \le n.$
 Then for every admissible multi index $\al$ of $P_{l+1}$, there exists a constant $M>0$ that depends on $l ,\vert \al\vert$ such that
 for  $r>0, \de>0$ with $r>M\de^{\frac{\vert \al\vert}{\vert \al\vert+1}}$
 we have
$$\vol_n (V_\de \cap \De_r^n) \le C r^{n-1+\theta} \de^{\tau}.$$
where $\theta:=\frac{1-2^{l-1}}{2^{l-1}+\vert \al\vert},\tau:=\frac{2^{l-1}}{2^{l-1}+\vert \al\vert}.$ Here $C>0$ is a constant depends only on $P_1, \cdots, P_{l+1}$.
\end{theorem}
\n
{\bf Remarks.} (i) The assumption that $\deg P_{l+1} \ge 1$ implies
$P_1, \cdots, P_l$ are algebraically independent in the sense that there does not exist a polynomial $Q \in \R [y_1, \cdots, y_l], Q \not \equiv 0$ satisfying
$Q(P_1, \cdots, P_l) \equiv 0$ on $\R^n$. Indeed, suppose otherwise, then after a linear change of coordinates in $\R^l$, we may assume that $Q$ is a monic polynomial of degree $d$ in $y_1.$
By differentiating the equation $Q(P_1, \cdots, P_l)=0$  with respect to variables $x_1, \cdots, x_l$ we obtain the "linear" system of equations
$$\sum_{j=1}^l \frac{\partial Q}{\partial y_j} (P_1 (x), \cdots, P_l (x)) \frac{\partial P_j}{\partial x_i} (x)=0, \ \forall 1 \le i \le l.$$
This implies that
$Q_1 (P_1, \cdots, P_l) \equiv 0$ on $\R^n \setminus P_{l+1}^{-1} (0),$
where $Q_1:=\frac{\partial Q}{\partial y_1}$. Since $P_{l+1} \not\equiv 0$ , we have $Q_1 (P_1, \cdots, P_l) \equiv 0$ entirely on $\R^n.$
Continuing this process $d$ times we obtain $Q_d:= \frac{\partial^d Q}{\partial y_1^d} \equiv 0$ which is absurd.

\n
(ii) Let $P$ be the polynomial given in (\ref{example}). By direct computations we get
$$\begin{aligned}
P_2(x_1,x_2)&=\fr{\partial P}{\partial x_1}=dx_1^{d-1}+x_2^m,\\
P_3 (x_1,x_2)&=\fr{\partial P}{\partial x_1}\fr{\partial^2 P}{\partial x_1\partial x_2}-
\fr{\partial P}{\partial x_2}\fr{\partial^2 P}{\partial x_1^2}=dm(2-d)x_1^{d-1}x_2^{m-1}-dp(d-1)x_1^{d-2}x_2^{p-1}.
\end{aligned}$$
Since $d \ge 3,$ we see that $x_1^{d-1}x_2^{m-1}$ is an admissible monomial for $P_3$.
By applying Theorem 3.2 to $P_3$ and this admissible monomial, 
we obtain constants $C>0, M>0$ such that for $\de>0$ and $r>M \de^{\fr{d+n-2}{d+n-1}},$ 
we have
$$\vol_n (V_\de \cap \De_r^n) \le  Cr^{1-\lambda}\de^{\lambda}.$$
where $\lambda:=\dfrac{1}{d+n}$.

\noindent The above estimate is clearly better than (\ref{estimate2}) if $r\de^{\fr{d+n}{1+n}}>1.$
\begin{proof}
In what follows, by $C_1, C_2, \cdots$ we mean constants that depend only on $P_1, \ldots, P_{n+1}.$
For $1 \le m \le n$ we define the polynomial maps
$$\Phi_m (x):= (P_1 (x),  \ldots, P_m (x), x_{m+1}, \ldots, x_n).$$
It is clear that  the Jacobian $J\Phi_m$ equals to $P_{m+1}$.
By Sard's theorem there exists a subset $X$ of Lebesgue measure $0$ in $\R^n$ and a constant $C_1>0$
such that for each $1 \le m  \le l$ we have
$$\# \{t \in \R^n:  \Phi_m (t)=\xi\} \le C_1, \quad \forall  \xi \in \R^n \setminus X.$$
For $\de_1>0, \cdots, \de_l>0$ we set $W_1:= V_\de$ and for each $2 \le m \le l+1$  
$$W_m:= \{x \in \R^n: \vert P_1 (x)\vert<\de, P_2 (x)<\de_1, \cdots, \vert P_m (x)\vert<\de_{m-1}\}.$$
By the area formula (cf. \cite{Federer1969}), for $1 \le m \le l$ we get
\begin{equation}
\int_{W_m \cap \De_r^n}\vert P_{m+1}\vert d\la_n=\int_{W_m \cap \De_r^n} \vert  J \Phi_m \vert  d\la_n =
\int _{(\Phi_m (W_m \cap \De_r^n)) \setminus X} \# \{t \in W_m \cap \De_r^n:  \Phi_m (t)=\xi\} d\la_n (\xi).\end{equation}
Since $\Phi_m (W_m \cap \De_r^n)$ is contained in the cube $(-\de, \de) \times (-\de_1, \de_1) \times \cdots \times (-\de_{m-1}, \de_{m-1}) \times (-r, r) \times \cdots \times (-r, r),$
we obtain
\begin{equation}
\vol_n (\Phi_m (W_m \cap \De_r^n)) \le 2^n r^{n-m} \de \de_1 \cdots \de_{m-1}.\end{equation}
Thus, from (10) and (11) we infer
$$\int_{W_m \cap \De_r^n}\vert P_{m+1}\vert d\la_n \le 2^n C_1 r^{n-m} \de \de_1 \cdots \de_{m-1}.$$
It follows that
$$\vol_n \{x \in W_m \cap \De_r^n: \vert P_{m+1} (x)\vert \ge \de_m\} \le 2^n C_1 r^{n-m} \de \frac{\de_1 \cdots \de_{m-1}}{\de_m}.$$
So for every $1 \le m \le l$ we have
\begin{equation}
\vol_n (W_m \cap \De_r^n)  \le C_2 \Big [\vol_n (W_{m+1} \cap \De_r^n)+r^{n-m}\de \frac{\de_1 \cdots \de_{m-1}}{\de_m}\Big ].\end{equation}
On the other hand, for $0<\de_l<r$,
by Theorem 3.1 and the relation between $r$ and $\de$ we obtain
\begin{equation}
\vol_n (W_{l+1} \cap \De_r^n) \le C_3 \Big [r^{n-1}\de_l^{1/\vert \al\vert}+ \de_l^{n/\vert \al\vert} \Big ] \le 2C_3 r^{n-1}\de_l^{1/\vert \al\vert}. \end{equation}
Summing up, from (12) and (13) we obtain the following estimate which holds true for every $\de_1>0, \cdots, \de_{l-1} >0, 0<\de_l<r$
\begin{equation}
\vol_n (V_\de \cap \De_r^n) \le C_4r^{n-1} f(\de_1, \cdots, \de_l)
\end{equation}
where   $$ f(\de_1, \cdots, \de_l)=\frac{\de}{\de_1}+r^{-1}\de\frac{\de_1}{\de_2}+\cdots+r^{1-l} \de \frac{\de_1 \cdots \de_{l-1}}{\de_l}+\de_l^{1/\vert \al\vert}.$$
It now suffices to choose $\de_1, \cdots, \de_l$ in the above range that minimizes $f$.
In the notation of Lemma 2.2  we find that
\begin{equation}
\min_{(\R^+)^l} f=C(l, \vert \al\vert) \va (\de, r^{-1} \de, \cdots, r^{1-l} \de)^{\theta+\tau}.
\end{equation}
Our first task is to write down precisely the term on the right hand side.
We have
$$\va (\de, r^{-1} \de, \cdots, r^{1-l} \de)= \de^{2^{l-2}} (r^{-1}\de)^{2^{l-3}} \cdots (r^{2-l} \de) (r^{1-l} \de)
=\de^{2^{l-1}} r^{\psi}.$$
where $\psi:=(1-l)+\sum_{j=2}^l (j-l) 2^{j-2}.$

\noindent Since
\begin{align*}
\psi=1-l+\sum_{j=2}^l (j-l) 2^{j-2} &= 1-l(1+\sum_{j=2}^l 2^{j-2})+\sum_{j=2}^l j2^{j-2}\\
&=1-l2^{l-1}+\frac1{4}\sum_{j=2}^l j2^j \\
&=1-l2^{l-1}+2^{l-1}(l-1)=1-2^{l-1}
\end{align*}
we obtain
$$ \va (\de, r^{-1} \de, \cdots, r^{1-l} \de)=r^{1-2^{l-1}}\de^{2^{l-1}}.$$
It follows that
$$\min_{(\R^+)^l} f =C(l, \vert \al) r^{\theta} \de^{\tau}.$$
where $\theta=\frac{1-2^{l-1}}{2^{l-1}+\vert \al\vert},\tau=\frac{2^{l-1}}{2^{l-1}+\vert \al\vert}.$

\noindent It is clear that $f$ attains its minimum at some point $(\de_1^0, \cdots, \de_l^0)$  with $\de_1^0>0, \cdots, \de_l^0>0.$
Finally, we must show that $\de_l^0<r.$
By Lemma 2.2, we have
$$\de_l^0= C' (l, \vert \al \vert) r^{\theta \vert \al\vert} \de^{\tau \vert \al \vert}.$$
This implies that $\de_l^0<r$ provided that $r>C'' (l, \vert \al\vert) \de^{\frac{\vert \al\vert}{\vert \al\vert+1}}.$
Putting this assertion together with (14) and (15) we arrived at the
desired estimate.
 \end{proof}
\section
{Estimations of certain integrals}
\vskip0,2cm
\n
In this final section, we will first apply the preceding volume estimates to get some lower bounds on integration index of real polynomials.
\begin{proposition}
Let $P$ be polynomials in $\R^n$ with $P(0)=0$ and $\al$ be an admissible index for $P$.
Then for every $\mu \in (0, \frac1{\vert \al\vert})$ and $r>0$ we have
$$\int_{\De_r^n} \frac1{\vert P(x)\vert^ \mu} d\la_n \le  Cr^{n-\mu \vert \al\vert}.$$
In particular $\i (P) \ge \fr1{ad(P)}.$
\end{proposition}
\begin{proof}
Fix $\mu \in (0, \frac1{\vert \al\vert})$ and set $s:= (\mu \vert \al\vert)^{-1}$. Then $s >1$. In what follows by $C, C', \cdots$ we mean positive constants which are independent of $r.$
Now for $\de>\frac1{2}r^{-1/s}$, using Theorem 3.1, we get
\begin{align*}
\va (\de, r):&=\vol_n \{x \in \De_r^n: \frac1{\vert P(x))\vert^\mu} \ge \de \}=
\vol_n \{x \in \De_r^n: \vert P (x)\vert \le \de^{-1/\mu}\}\\
& \le C [r^{n-1} \de^{-s}+ \de^{-ns}] \le C'r^{n-1} \de^{-s}.
\end{align*}
For any $a>\frac1{2}r^{-1/s}$, by the arrangement formula (see \cite[p. 421]{H-S1975}) we have
\begin{eqnarray*}
\int_{\De_r^n} \frac1{\vert P(x)\vert^\mu} d\la_n &=& \int_0^\infty \va(\de, r)d\de  =  \int_0^a \va (\de, r)d\de+\int_a^\infty \va (\de, r)d\de \\
& \le & 2^n r^n a+C' \int_a^{\infty} r^{n-1} \de^{-s}d\de \le C''
 [\underbrace{r^n a+r^{n-1} \frac{a^{1-s}}{s-1}}_{:=f(a)}].
\end{eqnarray*}
Since $f' (a)= r^n-r^{n-1} a^{-s}$ changes sign at $a_0:= r^{-1/s}$, it is easy to check that $f$ attains minimum at $a_0$. So by choosing
$a=a_0$, in view of the above inequalities we obtain
$$\int_{\De_r^n} \frac1{\vert P(x)\vert^\mu} d\la_n \le C'' f(a_0)= C''' r^{n-\frac1{s}}, \  \forall r>0.$$
This is the desired estimate.
\end{proof}
\noindent
{\bf Remarks.} (i) Consider again the case where $P$ is the polynomial given by (\ref{example}), i.e.,
$P(x_1,x_2)=x_1^d+x_1x_2^m+x_2^p$ with $d>m>p\ge 1.$ 
By Proposition 4.1 we have $\bold i(P) \ge \fr1{m+1}.$
On the other hand, using Newton polygon of $P$, it is possible to deliver the {\it actual} value of $\bold i(P).$ 
More precisely, consider the Newton polygon $N(P)$ of $P$ which is the convex hull of the closed set
$$\{x_1, x_2) \in \mathbb R^2: x_1 \ge d, x_2\ge 0\} \cup \{x_1, x_2) \in \mathbb R^2: x_1 \ge 1, x_2\ge m\}
\cup \{x_1, x_2) \in \mathbb R^2: x_1 \ge 0, x_2\ge n\}.$$
The Newton distance $ND(P)$ is now defined to be the least $\delta >0$ such that $(\delta, \delta) \in N(P).$
By a simple computation we get $ND(P)=\fr{dm}{d+m+1}.$ Thus, by the main theorem in \cite{Greenblatt2005}, we obtain
$\bold i(P)=\fr{d+m+1}{dm}.$ 

\noindent
(ii) In the higher dimensional case, i.e., $n \ge 3,$ by using some results in \cite{Greenblatt2010}, we will derive an {\it upper} bound for $\i (P)$ under the additional condition that $P(0)=\nabla P(0)=0.$ 
To this end, let $ND(P)$ be the Newton distance of $P$ which is equal to the least $\delta>0$ such that $(\delta, \cdots, \delta)$ belongs to $N(P),$ the Newton polyhedron of $P$.
Fix $\delta>\fr1{ND(P)}$ and let $r>0$ be an arbitrary number. We claim that
$$\int_{\De_r^n} \frac1{\vert P(x)\vert^\de} d\la_n=\infty.$$
Indeed, let $\va$ be a smooth function with compact support in the interior of $\De_r^n$ such that 
$0 \le \va \le 1, \va(0)=1.$
It follows that
$$\int_{\De_r^n} \frac1{\vert P(x)\vert^\de} d\la_n \ge \ve^{-\de} \vol_n\{x \in \De_r^n: \vert P(x)\vert<\ve\}
\ge \ve^{-\de}\int_{\{\vert P(x)\vert<\ve\}} \va(x)d\la_n=\infty.$$
Here the last assertion follows from 
Theorem 1.2 (a) in \cite{Greenblatt2010} which says that
the final integral is at least asymptotically $\vert \ln \ve\vert^m \ve^{\fr1{ND(P)}}$ as $\ve \to 0$, where $m$ is a non-negative integer. The claim now follows. 
\vskip0,2cm
\noindent
By the same reasoning where Theorem 3.2 is invoked instead of Theorem 3.1, we have the following result.
\begin{proposition}
Let $P$ be polynomials in $\R^n$ with $P(0)=0.$ Assume that $P$ satisfies the assumption in  Theorem 3.2  with some 
$l \ge 1.$ Let $\al$ be an admissible index for $P_{l+1}$.
Then for every $\mu \in (0, \tau)$, there exists a constant $C>0$  which depends only on coefficients of $P_1, \cdots, P_l$ and $\al$ such that
 $$\int_{\De_r^n} \frac1{\vert P(x)\vert^ \mu} d\la_n \le C\Big [r^{n-\mu \frac{\vert \al\vert+1}{\vert \al\vert}}+r^{n+\theta-\frac{\tau}{\mu}} \Big ].$$
 where $\theta=\frac{1-2^{l-1}}{2^{l-1}+\vert \al\vert},\tau=\frac{2^{l-1}}{2^{l-1}+\vert \al\vert}.$ In particular $\bold i(P) \ge\tau.$
\end{proposition}
\begin{proof}
The proof is quite similar to that of Proposition 4.1. The main point is to take care of the constrain between $r$ and $\de$ in Theorem 3.2.
For simplicity of notation, we set
$$\ga: =n-1+\theta.$$
Fix $\mu \in (0,  \tau)$ and set $s:= \tau/\mu$. Then $s >1$. In what follows by $C, C', \cdots$ we mean positive constants which are independent of $r.$
Now for $\de>\de_0:=Mr^{-\mu \frac{\vert \al\vert+1}{\vert \al\vert}}$, using Theorem 3.2, we get
$$\va (\de, r):=\vol_n \{x \in \De_r^n: \frac1{\vert P(x))\vert^\mu} > \de \}=
\vol_n \{x \in \De_r^n: \vert P (x)\vert<\de^{-1/\mu}\} \le Cr^{\ga} \de^{-s}.$$
By the arrangement formula (see \cite[p. 421]{H-S1975}) we have
\begin{eqnarray*}
\int_{\De_r^n} \frac1{\vert P(x)\vert^\mu} d\la_n &=& \int_0^\infty \va(\de, r)d\de  =  \int_0^{\de_0} \va (\de, r)d\de+\int_{\de_0}^\infty \va (\de, r)d\de \\
& \le & 2^n r^n \de_0+C \int_{\de_0}^{\infty} r^{\ga} \de^{-s}d\de \le C' [r^n \de_0+r^{\ga}\de_0^{1-s}].
\end{eqnarray*}
By plugging $\de_0$ into the last term we obtain the desired estimate.
\end{proof}
\noindent
{\bf Remark.} Consider again the polynomial $P$ given in (\ref{example}). As in the Remark (ii) that follows Theorem 3.2, we apply Proposition 4.2 to $P_3,$ i.e., $l=2$ and the admissible monomial $x_1^{d-1}x_2^{m-1}$ to get the lower bound
$\bold i(P) \ge \fr{2}{m+d}$. This bound is obviously better than the one given by Proposition 4.1 in the case where
$d=m+1.$
\vskip0,4cm
\noindent
Our last result deals with finiteness of certain oscillatory integrals. First we recall the following classical fact which is a cornerstone in estimating oscillatory integrals.
\begin{lemma}[van der Corput Lemma]
There exists an absolute constant $c > 0$ so that for any $a < b,$ any $k \ge 1,$ and any polynomial $f \in \mathbb{R}[t]$ satisfying $|f^{(k)}(t)| \ge \gamma$ on $(a, b),$ we have
$$\left| \int_a^b e^{i f(t)} dt \right| \le c \gamma^{-\frac{1}{k}}.$$
\end{lemma}
\begin{proof}
By \cite[Lemma 4]{Rogers2005}, it suffices to show the lemma in the case where $k = 1.$
Since $f$ is a polynomial, $f'$ is a polynomial of degree $\deg f - 1.$ Then the line $\mathbb{R}$ is made up at most $\deg f- 1$ intervals on each of which $f'$ is monotone. Then the conclusion follows from
\cite[Lemma 3]{Rogers2005}.
\end{proof}
\begin{proposition}
Let $P$ be a polynomial in $\mathbb R^n$.
Let
$$I_r (\lambda) := \int_{\De_r^n} e^{i \lambda P} d\la_n, \ \lambda \in \mathbb R.$$
Then there is an absolute constant $C > 0$ such that for $\vert \la\vert>1$ and $r>1$ we have
$$|I_r(\lambda)| \le Cr^{n - 1} |\lambda|^{-\frac1{ad(P)}}.$$
\end{proposition}
\noindent
{\bf Remark.} 
Using Proposition 4.14 in \cite{Carbery1999} (or rather its proof) we obtain the following estimate
$$|I_r(\lambda)| \le Cr^{n - 1} |\lambda|^{-\frac1{\deg P}}, \ \forall \vert \la \vert>1, \forall r>1,$$
where $C$ is a positive constant which does not depend on $\lambda.$
\begin{proof}
Fix an admissible polynomial $x_1^{\al_1} \cdots x_n^{\al_n}$ of $P$.
After a permutation of coordinates we may write
$$P(x', x_n) := x_n^{\al_n} P_0 (x') + P_1(x') x_n^{d - 1} + \cdots + P_{\al_n} (x'),$$
where $\al_n \ge 1$ and $x_1^{\al_1} \cdots x_{n-1}^{\al_{n-1}}$ is an admissible polynomial for $P_0(x')$.
We will now use a reasoning similar to the one given in the proof of \cite[Proposition 4.14]{Carbery1999}.
For $\de \in (0, 1),$ we set
$$W_\de:= \{x' \in \De_r^{n-1}: \vert P_0 (x')\vert  < \de\}, W'_\de :=  \{x' \in \De_r^{n-1}: \vert P_0 (x')\vert  \ge \de\}.$$
By Fubini's theorem we obtain
\begin{eqnarray*}
|I_r(\lambda)|  &\le& \int_{W'_\de}  \left |\int_{-r}^r e^{i \lambda P(x', x_n)} dx_n \right|  dx' + \int \int_{W_\de \times [-r, r]}  e^{i \lambda P(x)}  dx'dx_n.
\end{eqnarray*}
From now on, by $C_1, C_2, \cdots$ we mean positive constants  independent of $\la, r, \de.$
By the van der Corput Lemma (Lemma 4.3) we get
$$\left | \int_{-r}^r e^{i \lambda P(x',x_n)} dx_n \right| \le C_1 (|\lambda|\de)^{-1/\alpha_n},  \forall x' \in W'_\de.$$
This implies that
$$ \int_{W'_\de}  \left |\int_{-r}^r e^{i \lambda P(x',x_n)} dx_n \right|  dx' \le C_1 (|\lambda|\de)^{-1/\alpha_n}\vol_{n-1} (W'_\de).$$
On the other hand, we note the following trivial estimate
$$\int \int_{W_\de \times [-r, r]}  e^{i \lambda P(x',x_n)} dx'd x_n \le 2r\vol_{n-1} (W_\de).$$
Summing up we arrive at
\begin{align*}
|I_r(\lambda)| &\le C_2 \Big [ (|\lambda|\de)^{-1/\alpha_n}\vol_{n-1} (W'_\de)+r\vol_{n-1} (W_\de)\Big ] \\
&=C_2 \Big [(\vert \la\vert \de)^{-1/\al_n} [(2r)^{n-1}-\vol_{n-1} (W_\de)]+r\vol_{n-1} (W_\de) \Big]\\
&=C_2 \Big [ (2r)^{n-1} (\vert \la\vert \de)^{-1/\al_n} + [r-(\vert \la\vert \de)^{-1/\al_n}] \vol_{n-1} (W_\de) \Big].
\end{align*}
Consider the new parameter $x: = (\vert \la\vert \de)^{-1/\al_n}$.
In order to dominate the last term, we need a volume estimate for $W_\de$ and the extra condition that $0<x<r.$
More precisely, since $0<\de<1<r$, by Theorem 3.1 we have
$$\vol_{n-1} (W_\de) \le C_3 r^{n-2} \de^{1/\vert \al'\vert} = C_3 r^{n-2} \vert \la\vert^{\kappa}
 x^{\kappa\al_n},$$
 where $\al' =(\al_2, \cdots, \al_n),\kappa:=-1/\vert \al'\vert.$ 
It follows, for $x \in (0, r)$,  that
$$\vert I_r (\la)\vert \le C_4 r^{n-2} \Big [ \underbrace{rx+(r-x)\vert \la\vert^{\kappa} x^{\kappa\al_n}}_{:=f(x)}\Big].$$
An easy computation gives
$$f' (x)= r- \vert \la\vert^{\kappa} \Big [(\kappa \al_n+1) x^{\kappa\al_n}-
r\kappa\al_nx^{\kappa\al_n-1} \Big].$$
Clearly the equation $f' (x)=0$ has a unique positive root at which $f$ attains its minimum.
We "approximate" this root by taking $x=x_0: = \vert \la\vert^{-\frac1{\al_n+\vert \al'\vert}}$. Obviously $0<x_0<1<r$ and the corresponding value $\de_0= \vert \la\vert^{-\frac{\vert \al'\vert}{\al_n+\vert \al'\vert}} \in (0, 1)$.
Thus our choice of $x$ (and therefore of $\de$) works. It follows that
$$\vert I_r (\la)\vert \le C_4 r^{n-2} f(x_0) \le Cr^{n-1} \vert \la\vert^{-\frac1{\al_n+\vert \al'\vert}}
= Cr^{n-1} \vert \la\vert^{-\frac1{\vert \al\vert}}.$$
Finally, by letting $x^\al$ runs over the set of admissible monomials and noting that $\vert \la\vert>1$, we obtain the desired estimate.
\end{proof}

\end{document}